\documentclass[12pt]{amsart}

\usepackage{amsfonts,amsmath,latexsym,amssymb,verbatim,amsbsy,times}
\usepackage{amsthm}
\usepackage{pstricks}

\theoremstyle{plain}
\newtheorem{THEOREM}{Theorem}[section]

\newtheorem{lemma}[THEOREM]{Lemma}
\newtheorem{proposition}[THEOREM]{Proposition}

\theoremstyle{definition}

\theoremstyle{remark}

\newcommand{\Z}{\ensuremath{\mathbb{Z}}}   
\newcommand{\R}{\ensuremath{\mathbb{R}}}   
\newcommand{\T}{\ensuremath{\mathbb{T}}}   

\def \RR {\mathbb{R}}

\def \d {\delta}

\def \e {\epsilon}
\def \f {\varphi}

\def \l {\lambda}

\def \n {\nabla}

\def \ex {\vec{e}_1}
\def \ey {\vec{e}_2}

\def \< {\langle}
\def \> {\rangle}
\def \p {\partial}
\def \ra {\rightarrow}
\def \ss {\subset}

\newcommand{\der}[2]{(#1 \cdot \nabla) #2}

\DeclareMathOperator{\diver}{div} %

\newcommand{\tri}[3]{#1 \otimes #2 : \n #3}

\begin{document}

\title[Ill-posedness in Besov spaces]{Ill-posedness of basic equations of fluid dynamics in Besov spaces}

\author{A. Cheskidov}
\thanks{The work of A. Cheskidov is partially supported by NSF grant DMS--0807827}
\address[A. Cheskidov and R. Shvydkoy]
{Department of Mathematics, Stat. and Comp. Sci.\\
M/C 249,\\
       University of Illinois\\
       Chicago, IL 60607}
\email{acheskid@math.uic.edu}

\author{R. Shvydkoy}
\thanks{The work of R. Shvydkoy was partially supported by NSF grant
DMS--0604050}

\email{shvydkoy@math.uic.edu}

\begin{abstract}We give a construction of a divergence-free vector field $u_0 \in H^s \cap B^{-1}_{\infty,\infty}$, for all $s<1/2$,  such that any Leray-Hopf solution to the Navier-Stokes equation starting from $u_0$ is discontinuous at $t=0$ in the metric of $B^{-1}_{\infty,\infty}$. For the Euler equation a similar
result is proved in all Besov spaces $B^s_{r,\infty}$ where $s>0$ if $r>2$, and $s>n(2/r-1)$ if $1 \leq r \leq 2$.
\end{abstract}


\maketitle

\section{Introduction}

In recent years numerous results appear in the literature on well-posedness theory of the Euler and Navier-Stokes equations in Besov spaces (see for example, \cite{amann,cannone,chae,pp,vishik} and references therein). The best local existence and uniqueness result known for the Euler equation states that for any initial condition $u_0 \in B^{\frac{n}{r}+1}_{r,1}$ with $1<r \leq \infty$, where $n$ is the dimension of the fluid domain, there exists a unique weak solution $u$ in space $C([0,T];B^{\frac{n}{r}+1}_{r,1})$, for some $T>0$, such that $u(t) \ra u_0$ in $B^{\frac{n}{r}+1}_{r,1}$.  The case of $r=2$, $n=3$ is especially interesting for it constitutes the borderline space for applicability of the standard energy method in proving local well-posedness (see \cite{mb}). Notice that $B^{5/2}_{2,1}$ is a proper subspace of the Sobolev space $H^{5/2} = B^{5/2}_{2,2}$, where local existence is an outstanding open problem. As a part of a construction presented here in Proposition~\ref{p-euler} we show that the Euler equation is ill-posed in the opposite extreme space with respect to summation, namely in $B^{5/2}_{2,\infty}$. Specifically, there exists a $u_0 \in B^{5/2}_{2,\infty}$ such that any energy bounded weak solution to the Euler equation that starts from $u_0$ does not converge back to $u_0$ is the metric of $B^{5/2}_{2,\infty}$ as time goes to zero. Another particular case of Proposition~\ref{p-euler} demonstrates similar ill-posedness  result in $B^{1}_{\infty,\infty}$ thus precluding a possible extension of Pak and Park's result in $B^{1}_{\infty,1}$ (see \cite{pp}).

In the second part of this note we address the question of ill-posedness for the Navier-Stokes equations in the critical Besov space $X=B^{-1}_{\infty,\infty}$.
We recall that the homogeneous space $\dot{X}=\dot{B}^{-1}_{\infty,\infty}$ is invariant with respect to the natural scaling of the equation in $\R^n$. Moreover it is the largest such space \cite{cannone}. The non-homogeneous space considered in this note is even larger although (quasi-)invariant only with respect to the small
scale dialations. In a recent work of Bourgain and Pavlovic \cite{bp} the authors constructed a mild solution to NSE with initial condition $\|u_0\|_{\dot{X}} < \d$ such that at a time $t<\d$ the solution satisfies $\|u(t)\|_{\dot{X}} > 1/\d$. This shows the evolution under NSE is not continuous from $\dot{X}$ to $C([0,T]; \dot{X})$. In our Proposition~\ref{p-nse}, similar to the case of the Euler equation, we construct an initial condition $U$ which belongs to all Besov spaces $B^{3/r-1}_{r,\infty}$ in the range $1<r\leq \infty$, -- in particular $U$ has finite energy -- such that any Leray-Hopf weak solution starting from $U$ does not return to $U$ in the metric of inhomogeneous space $X$. This demonstrates an even more dramatic breakdown of NSE evolution in $X$  as there is no continuous trajectory in $X$ at all. More importantly our construction gives a simple model for the forward energy cascade, which is typically observed in turbulent flows \cite{frisch}. Incidentally, the result proved in \cite{cs} shows that any left-continuous Leray-Hopf solution in $X$ is necessarily regular.

We consider periodic boundary conditions for two main reason. Firstly, we do not make use of lower frequencies in our analysis, and secondly,
our constructions become much more transparent. However with the technique developed in \cite{ccfs} the results can be carried over to the open space too.

Let us now introduce the notation and spaces used in this paper.
We will fix the notation for scales $\lambda_q = 2^q$ in some inverse length
units. Let us fix a nonnegative radial function $\chi \in {C_0^{\infty}}(\R^n)$ such that $\chi(\xi)=1$ for $|\xi|\leq 1/2$, and $\chi(\xi) = 0$ for $|\xi|\geq 1$. We define
$\f(\xi) = \chi(\l_1^{-1}\xi) - \chi(\xi)$,
and $\f_q(\xi) = \f(\l_q^{-1}\xi)$ for $q \geq 0$, and $\f_{-1} = \chi$.
For a tempered distribution vector field $u$ on the torus $\T^n$ we consider the Littlewood-Paley projections
\begin{equation}
u_q (x) = \sum_{k \in \Z^n} \hat{u}(k) \f_q(k) e^{i k\cdot x}, \quad q \geq -1.
\end{equation}
So, we have $u = \sum_{q=-1}^\infty u_q$
in the sense of distributions. We also use the following notation
$u_{\leq q} = \sum_{p=-1}^q u_p$, and $\tilde{u}_q = u_{q-1} + u_q + u_{q+1}$.

Let us recall the definition of Besov spaces. A tempered distribution $u$ belongs to $B^s_{r,l}$ for $s\in \RR$, $1\leq l,r \leq \infty$ iff
$$
\|u\|_{B^s_{r,l}} = \left(\sum_{q\geq -1} (\l_q^s \|u_q\|_r)^l \right)^{1/l} < \infty.
$$

\section{Inviscid case}
The Euler equation for the evolution of ideal fluid is given by
\begin{equation}\label{ee}
u_t + \der{u}{u} = -\n p,
\end{equation}
where $u$ is a divergence free field on $\T^n$. By a weak solution to \eqref{ee} we understand an $L^2$-valued weakly continuous field $u$ satisfying \eqref{ee} in the distributional sense. Let us recall that all such solutions have absolutely continuous in time Fourier coefficients (see for example \cite{s}).

Our construction below is two-dimensional. So, we denote by $\ex,\ey$ the vectors of the standard unit basis and define
$$
u_0(x,y) = \ex \cos(y) + \ey \sum_{q=0}^\infty \frac{1}{\l_q^s} \cos(\l_q x).
$$
\begin{proposition}\label{p-euler} If $u$ is a weak solution to the Euler equation \eqref{ee} with initial condition $u(0) = u_0$. Then there is $\delta=\delta(n,r,s)>0$ independent of $u$ such that
we have
\begin{equation}\label{e:main}
\limsup_{t\ra 0^+} \| u(t) - u_0\|_{B^s_{r,\infty}} \geq \delta,
\end{equation}
where $s>0$ if $r>2$, and $s>n(2/r-1)$ if $1 \leq r \leq 2$.
\end{proposition}
The rest of the section is devoted to the proof of Proposition~\ref{p-euler}.

Let us denote $X = B^s_{r,\infty}$. We can make the assumption that for some $t_0>0$, $u \in L^\infty([0,t_0]; X)$. Indeed, otherwise \eqref{e:main} follows immediately. Further proof is based on the fact that $u_0$ produces a strong forward energy transfer which forces $u$ to actually escape from $B^s_{r,\infty}$ unless \eqref{e:main} is met. To this end, let us consider frequencies $\xi_q = (\l_q,1)$. Let $p(\xi)$ be the symbol of the Leray-Hopf projection. By a direct computation we have
\begin{equation}\label{transfer}
f_q = p(\xi_q) (u_0 \cdot \n u_0)^\wedge (\xi_q ) = i \l_q^{1-s} \ey + O(1/\l^s_q).
\end{equation}
We will prove the following estimate for the nonlinear term
\begin{equation}\label{e:basic}
|(u\cdot \n v)_q|_1 \lesssim \l_q^{1-s} \|u\|_{X} \|v\|_{X},
\end{equation}
for all $u,v \in X$ and $q \geq -1$. First, let us assume that $r\leq 2$. Using the identity $\diver(u \otimes v) = u \cdot \n v$
and the Bernstein inequality we obtain
\begin{align}
|\diver(u \otimes v)_q|_1 & \lesssim  \l_q |(u \otimes v)_q |_1
\leq \l_q \sum_{\substack{p',p'' \geq q \\  |p' - p''| \leq 2}} |u_{p'}|_r |v_{p''}|_{r'} \label{est1} \\
&+ \l_q |u_q|_r \sum_{p \leq q} |v_{p} |_{r'} + \l_q |v_q|_r \sum_{p \leq q} |u_p|_{r'} \label{est2}.
\end{align}
Using that
$$
|w_p|_{r'} \lesssim \l_p^{n(2/r-1)} |w_p|_{r},
$$
we have for the first sum
\begin{multline*}
\l_q \sum_{\substack{p',p'' \geq q \\ |p' - p''| \leq 2}} |u_{p'}|_r |v_{p''}|_{r'} \lesssim  \l_q \sum_{\substack{p',p'' \geq q \\ |p' - p''| \leq 2}} |u_{p'}|_r \l_{p'}^s |v_{p''}|_{r} \l_{p''}^s \l_{p''}^{n(2/r-1) -2s}  \\
\lesssim \l_q^{1 + n(2/r-1) -2s } \|u\|_X \|v\|_X.
\end{multline*}
For the second sum we obtain
$$
\l_q |u_q|_r \sum_{p \leq q} |v_{p} |_{r'} \lesssim  \l_q^{1-s}\l_q^s |u_q|_r \sum_{p \leq q} |v_{p} |_{r} \l_p^s \l_{p}^{n(2/r-1) - s} \lesssim \l_q^{1-s} \|u\|_X \|v\|_X.
$$
Similar estimate holds for the third term. We thus obtain \eqref{e:basic}.

In the case $r>2$, we use the basic embedding $L^{r} \ss L^{r'}$ instead of Bernstein's inequalities in \eqref{est1}--\eqref{est2}. The rest of the argument is similar.

We have
\begin{equation}\label{mild}
\hat{u}(\xi_q,t) = \hat{u}(\xi_q,0) + \int_0^t p(\xi_q)(u\cdot \n u)^\wedge(\xi_q,s) ds,
\end{equation}
for all $t>0$. By our construction, $\hat{u}(\xi_q,0) =0$. On the other hand we can estimate using \eqref{e:basic}
\begin{align*}
| p(\xi_q)(u\cdot \n u)^\wedge(\xi_q,s) - f_q | & \leq | (u\cdot \n u)^\wedge(\xi_q,s) -
(u_0 \cdot \n u_0)^\wedge(\xi_q) | \\
& = | (u\cdot \n u)_q^\wedge(\xi_q,s) - (u_0 \cdot \n u_0)_q^\wedge(\xi_q) | \\
& \leq | (u\cdot \n u)_q (s) - (u_0 \cdot \n u_0)_q |_1 \\
&\lesssim  \l_q^{1-s}(\|u(s)\|_{X} + \|u_0\|_{X}) \|u(s) - u_0\|_{X}.
\end{align*}
Thus, from \eqref{mild} we obtain
$$
\l_q^s |\hat{u}(\xi_q,t)| \geq t \l_q- tO(1) -  C \l_q \int_0^t(\|u(s)\|_{X} + \|u_0\|_{X}) \|u(s) - u_0\|_{X} ds.
$$
We can see that if the limit in \eqref{e:main} does not exceed $\delta = 1/(10C)$ then the integral becomes less than $t/2$. This implies that $u(t) \notin X$.

\section{Ill-posedness of NSE}

\def \lq {\l_{q_j}}
\def \lqm { \l_{q_j - 1} }
\def \lqk { \l_{q_k} }

Now we turn to the analogous question for the viscous model. The Navier-Stokes equation is given by
\begin{equation}\label{nse}
u_t + \der{u}{u} = \nu \Delta u - \n p.
\end{equation}
Here $u$ is a three dimensional divergence free field on $\T^3$. We refer to \cite{temam} for the classical well-posedness theory for this equation.
Let us recall that for every field $U\in L^2(\T^3)$ there exists a weak solution $u \in C_w([0,T);L^2) \cap L^2([0,T);H^1)$ to \eqref{nse} such that the energy inequality
\begin{equation}\label{enineq}
|u(t)|_2^2 + 2\nu \int_0^t |\n u(s)|_2^2 ds \leq |U|_2^2,
\end{equation}
holds for all $t>0$ and $u(t) \ra U$ strongly in $L^2$ as $t\ra 0$. In what follows we do not actually use inequality \eqref{enineq} which allows
us to formulate a more general statement below in Proposition~\ref{p-nse}.

Let us fix a small $\e >0$. Let us choose a sequence $q_1<q_2<...$ with elements sufficiently far apart so that
$\l_{q_i}^2 / \l_{q_{i+1}} < \e$. Let us fix a small $c>0$ and consider the following integer lattice blocks:

\begin{align*}
A_j & = [(1-c) \lq, (1+c) \lq] \times [-c \lq, c\lq]^2 \cap \Z^3 \\
B_j & = [-c \lqm, c \lqm]^2 \times [(1-c) \lqm, (1+c) \lqm ] \cap \Z^3\\
C_j& = A_j + B_j \\
A_j^*& = -A_j, \ B_j^* = - B_j, \ C_j^* = -C_j.
\end{align*}
Thus, $A_j$, $C_j$ and their conjugates lie in the $\lq$-th shell, while $B_j$, $B_j^*$ lie in the contiguous $\lqm$-th shell. Let us denote
$$
\ex(\xi) = p(\xi) \ex \text{ and } \ey(\xi) = p(\xi) \ey.
$$
We define
\begin{equation}
U = \sum_{j \geq 1} (U_{q_j} + U_{q_j - 1} ),
\end{equation}
where
$$
\widehat{U_{q_j}} =\frac{1}{\lq^2} \left( \ey(\xi) \chi_{A_j \cup A_j^*} + i (\ey(\xi) - \ex(\xi))\chi_{C_j} - i (\ey(\xi) - \ex(\xi))\chi_{C_j^*}    \right),
$$
and
$$
\widehat{U_{q_j - 1}} = \ex(\xi) \chi_{B_j \cup B_j^*}.
$$
Since $U$ has no modes in the $(q_j + 1)$-st shell, $\tilde{U}_{q_j} = U_{q_j - 1} + U_{q_j}$.

\begin{lemma}\label{ubesov}
We have $U \in B^{\frac{3}{r} - 1}_{r,\infty}$, for all $1<r \leq \infty$.
\end{lemma}
\begin{proof}
We give the estimate only for one block. Using boundedness of the Leray-Hopf projection, we have for $1<r<\infty$
$$
| \lq^{-2} (\ey(\cdot) \chi_{A_j})^{\vee} |_r \lesssim \lq^{-2} | (\chi_{A_j})^\vee |_r \leq \lq^{-2} | D_{c\lq} |_r^3,
$$
where $D_N$ denote the Dini kernel. By a well-known estimate, we have $|D_N|_r \leq N^{1-\frac{1}{r}}$, which implies
the lemma.

If $r = \infty$, we simply use the triangle inequality to obtain
$$
|U_{q_j}|_\infty \lesssim \lq.
$$
\end{proof}
Let us now examine the trilinear term. We will use the following notation for convenience
\begin{equation}
u\otimes v : \n w = \int_{\T^3} v_i \p_i w_j u_j dx.
\end{equation}

Using the antisymmetry we obtain
\begin{align*}
\tri{U}{U}{U_{q_j}}& = \sum_{k \geq j+1} \tri{\tilde{U}_{q_k}}{\tilde{U}_{q_k}}{U_{q_j}} +
\tri{\tilde{U}_{q_j}}{\tilde{U}_{q_j}}{U_{q_j}} \\
&+ \tri{U_{\leq q_{j-1}}}{\tilde{U}_{q_j}}{U_{q_j}} + \tri{\tilde{U}_{q_j}}{U_{\leq q_{j-1}}}{U_{q_j}} \\
&= \sum_{k \geq j+1} \tri{\tilde{U}_{q_k}}{\tilde{U}_{q_k}}{U_{q_j}} + \tri{U_{q_j - 1}}{U_{q_j}}{U_{q_j}} -
\tri{U_{q_j}}{U_{q_j}}{U_{\leq q_{j-1}}} \\
& = A+B+C.
\end{align*}
Using Bernstein's inequalities we estimate
\begin{align*}
|A| & \lesssim  \lq |U_{q_j}|_\infty \sum_{k \geq j+1} | \tilde{U}_{q_k} |_2^2 \lesssim \frac{\lq^2}{\l_{q_{j+1}}} \leq \e, \\
|C| & \lesssim |U_{q_j}|_2^2 \sum_{k \leq j-1} \lqk | \tilde{U}_{q_k} |_\infty \lesssim \frac{\l_{q_{j-1}}^2}{\lq} \leq \e.
\end{align*}
On the other hand, a straightforward computation show that
\begin{equation}
B \sim \lq.
\end{equation}
\begin{proposition}\label{p-nse}
Let $u\in C_w([0,T); L^2) \cap L^2([0,T); H^1)$ be a weak solution solution to the NSE with initial condition $u(0) = U$. Then there is $\delta=\delta(u) >0$ such that
\begin{equation}
\limsup_{t \ra 0+} \|u(t) - U\|_{B^{- 1}_{\infty,\infty}} \geq \delta.
\end{equation}
If in addition $u$ is a Leray-Hopf solution satisfying the energy inequality \eqref{enineq}, then $c$ can be chosen independent of $u$.
\end{proposition}
\begin{proof}
Using $u_{q_j}$ as a test function we can write
$$
\p_t( \tilde{u}_{q_j} \cdot u_{q_j}) = - \nu \n \tilde{u}_{q_j} \cdot \n u_{q_j} + \tri{u}{u}{u_{q_j}}.
$$
Denoting $E(t) = \int_0^t |\n u|_2^2 ds$ we obtain
\begin{multline}\label{crucial}
|\tilde{u}_{q_j}(t)|_2^2 \geq |U_{q_j}|_2^2 - \nu E(t) + c_1 \lq t \\
- c_2\int_0^t \left|  \tri{u}{u}{u_{q_j}} - \tri{U}{U}{U_{q_j}}  \right| ds,
\end{multline}
for some positive constants $c_1$ and $c_2$.
We now show that if the conclusion of the proposition fails then for some small $t>0$ the integral term is less than $c_1\lq t/ 2$ uniformly for all large $j$.
This forces $|\tilde{u}_{q_j}(t)|_2^2  \gtrsim \lq t$ for all large $j$. Hence $u$ has infinite energy, which is a contradiction.

So suppose that for every $\delta >0$ there exists $t_0 = t_0(\delta) >0$ such that $\|u(t)- U\|_{B^{- 1}_{\infty,\infty}} < \d$ for all $0<t\leq t_0$.
Denoting $w = u-U$ we write
\begin{multline*}
\tri{u}{u}{u_{q_j}} - \tri{U}{U}{U_{q_j}} = \tri{w}{U}{U_{q_j}} + \tri{u}{w}{U_{q_j}} \\ + \tri{u}{u}{w_{q_j}} = A + B+ C.
\end{multline*}
We will now decompose each triplet into three terms according to the type of interaction (c.f. Bony \cite{bony}) and estimate each of them separately.
\begin{multline*}
A = \sum_{\substack{p',p'' \geq q_j \\ |p' - p''| \leq 2}} \tri{w_{p'}}{U_{p''}}{U_{q_j}} + \tri{w_{\leq q_j}}{\tilde{U}_{q_{j}}}{ U_{q_j}} \\
+ \tri{\tilde{w}_{q_j}}{U_{\leq q_j}}{U_{q_j}} - \text{repeated} = A_1 + A_2 + A_3.
\end{multline*}
Using Lemma~\ref{ubesov} along with H\"{o}lder and Bernstein inequalities  we obtain
\begin{align*}
| A_1 | & \leq |\n U_{q_j}|_4 \sum |w_{p'}|_\infty |U_{p''}|_{4/3} \lesssim \lq^{5/4} \sum |w_{p'}|_\infty \l_{p''}^{-5/4} \lesssim  \d \lq, \\
|A_2| & = |\tri{U_{q_j}}{\tilde{U}_{q_{j}}}{w_{\leq q_j} }| \leq |\tilde{U}_{q_j}|_2^2 |\n w_{\leq q_j}|_\infty \lesssim
\lq^{-1} \sum_{p\leq q_j} \l_p^{2} \l_p^{-1} |w_p|_\infty < \d \lq, \\
|A_3| & \leq \lq |U_{\leq q_j}|_2 |U_{q_j}|_{2} |\tilde{w}_{q_j}|_\infty \lesssim  |\tilde{w}_{q_j}|_\infty < \d \lq.
\end{align*}
We have shown the following estimate:
\begin{equation}\label{A}
|A| \lesssim \d \lq.
\end{equation}
As to $B$ we decompose analogously,
\begin{multline*}
B = \sum_{\substack{p',p'' \geq q_j \\ |p' - p''| \leq 2}} \tri{u_{p'}}{w_{p''}}{U_{q_j}} + \tri{u_{\leq q_j}}{\tilde{w}_{q_{j}}}{ U_{q_j}} \\
+ \tri{\tilde{u}_{q_j}}{w_{\leq q_j}}{U_{q_j}} - \text{repeated} = B_1+B_2+B_3.
\end{multline*}
Again, using Lemma~\ref{ubesov}, Bernstein and H\"{o}lder inequalities we obtain
\begin{align*}
|B_1| & \lesssim \lq |U_{q_j}|_{2} \sum |u_{p'}|_2 |w_{p''}|_\infty \leq \d \lq^{1/2} |\n u|_2.\\
|B_2| & =  \left| \tri{U_{q_j}}{\tilde{w}_{q_{j}}}{u_{\leq q_j}} \right| \leq  |U_{q_j}|_{2} |\tilde{w}_{q_{j}}|_\infty |\n u_{\leq q_j}|_2 \\
     & \leq \lq^{-1/2} |\tilde{w}_{q_{j}}|_\infty |\n u|_2 \leq \d \lq^{1/2} |\n u|_2.\\
|B_3| & \leq |\tilde{u}_{q_j}|_2 |w_{\leq q_j}|_\infty |\n U_{q_j}|_2 \lesssim \lq^{1/2} |\tilde{u}_{q_j}|_2  \sum_{p \leq q_j}  \l_p^{- 1} |w_p|_\infty \l_p \\
&\lesssim \d \lq^{1/2} |\n u|_2.
\end{align*}
We thus obtain
\begin{equation}\label{B}
|B| \lesssim \d \lq^{1/2}|\n u|_2.
\end{equation}
Continuing in a similar fashion we write
\begin{multline*}
C = \sum_{\substack{p',p'' \geq q_j \\ |p' - p''| \leq 2}} \tri{u_{p'}}{u_{p''}}{w_{q_j}} + \tri{u_{\leq q_j}}{\tilde{u}_{q_{j}}}{w_{q_j}} \\
+ \tri{\tilde{u}_{q_j}}{u_{\leq q_j}}{w_{q_j}} - \text{repeated} = C_1+C_2+C_3.
\end{multline*}
\begin{align*}
|C_1| & \leq |\n w_{q_j}|_\infty \sum_{p \geq q_j - 2} |\tilde{u}_p|^2_2 \lesssim \lq^{1} |w_{q_j}|_\infty \lq^{-2} |\n u|_2^2 \leq \d |\n u|_2^2,\\
|C_2| & \leq | \n u|_{2} |\tilde{u}_{q_{j}}|_2 |w_{q_j}|_\infty \lesssim \lq^{-1} |\n u|_2^2 |w_{q_j}|_r \leq \d |\n u|_2^2.\\
\intertext{Now using a uniform bound on the energy $|u(t)|^2_2 \lesssim 1$
for almost all $t\geq0$, we estimate }
|C_3| &\lesssim \lq |w_{q_j}|_\infty |\tilde{u}_{q_j}|_2 \leq \d \lq |\n \tilde{u}_{q_j}|_2.
\end{align*}
Thus,
\begin{equation}\label{C}
|C| \lesssim \d |\n u|_2^2 + \d \lq |\n \tilde{u}_{q_j}|_2.
\end{equation}
Now combining estimates \eqref{A}, \eqref{B}, \eqref{C} along with the boundedness of $E(t_0)$ we obtain
\begin{multline} \label{nonlinearterm}
\int_0^{t_0} \left|  \tri{u}{u}{u_{q_j}} - \tri{U}{U}{U_{q_j}}  \right| ds \lesssim \d \lq t_0 + \d \lq^{1/2} t^{1/2}_0 \\ + \d +  \d \lq \int_0^{t_0} |\n \tilde{u}_{q_j}(s)|_2 ds.
\end{multline}
Using that
$$\int_0^{t_0} |\n \tilde{u}_{q_j}(s)|_2 ds \ra 0$$ as $j \ra \infty$
we can chose $\delta$ small enough and $j_0$ large enough so that the left hand side of the \eqref{nonlinearterm} is less than
$$\frac{c_1}{2c_2} \lq t_0$$
for all $j \geq j_0$.
Going back to \eqref{crucial} this implies
$$
|\tilde{u}_{q_j}(t_0)|_2^2 \geq |U_{q_j}|_2^2 - \nu E(t_0) + c_1 \lq t_0/2,
$$
for all $j >j_0$, which shows that $u(t_0)$ has infinite energy, a contradiction.

The last statement of the proposition follows from the fact that we have the bounds on $|u(t)|_2 \leq |U|_2$ and $E(t_0) \leq (2\nu)^{-1} |U|_2^2$ which remove dependence of the constants on $u$.
\end{proof}


\end{document}